\documentclass[a4paper]{amsart}

\usepackage{amsmath, amsthm, amssymb}
\usepackage[english]{babel}
\usepackage[utf8]{inputenc}

\usepackage[style=numeric, minnames=3,
            doi=false, url=false, isbn=false,
            firstinits=true,
            sortcites=true,
            backend=biber]{biblatex}      

\usepackage{todonotes}                    
\usepackage[all]{xy}                      
\usepackage[user]{zref}                   
\usepackage{chngcntr}                     
\usepackage{ifthen}                       

\newtheoremstyle{convenientthm}%
  {3pt}
  {3pt}
  {\itshape}
  {}
  {\bfseries}
  {.}
  {.5em}
  {\thmnumber{#2 }\thmname{#1}\thmnote{. #3}}
\theoremstyle{convenientthm}

\newtheorem*{theorem*}{Theorem}
\newtheorem*{proposition*}{Proposition}
\newtheorem*{lemma*}{Lemma}

\newtheoremstyle{convenientplain}%
  {3pt}
  {3pt}
  {}
  {}
  {\bfseries}
  {.}
  {.5em}
  {\thmnumber{#2 }\thmname{#1}\thmnote{. #3}}
\theoremstyle{convenientplain}

\newtheorem*{example*}{Example}

\def\al{\alpha}
\def\be{\beta}
\def\ga{\gamma}

\def\ep{\varepsilon}

\def\et{\eta}
\def\th{\theta}

\def\ph{\varphi}

\def\ps{\psi}

\def\Ga{\Gamma}
\def\De{\Delta}

\def\inv{^{-1}}
\def\x{\times}
\def\p{\partial}

\def\R{{\mathbb R}}

\let\on=\operatorname

\let\mb=\mathbb
\let\mc=\mathcal
\let\mf=\mathfrak

\newcommand{\ud}{\,\mathrm{d}}

\renewcommand*\thesection{\arabic{section}}

\counterwithin{equation}{subsection}
\renewcommand*\theequation{\arabic{equation}}

\makeatletter
\def\@seccntformat#1{%
  \protect\textup{%
    \protect\@secnumfont
    \expandafter\protect\csname format#1\endcsname 
    \csname the#1\endcsname
    \protect\@secnumpunct
  }%
}

\newcounter{keepeqno}
\newenvironment{enumeq}
 {\setcounter{keepeqno}{\value{equation}}%
  \begin{list}{(\theequation)}{\usecounter{equation}}%
  \setcounter{equation}{\value{keepeqno}}}
 {\end{list}}

\makeatletter
\zref@newprop{section}{\thesection}
\zref@newprop{subsection}{\thesubsection}
\zref@newprop{equation}{\theequation}
\zref@addprops{main}{section,subsection,equation}

\renewcommand{\eqref}[1]
{%
  \ifthenelse%
  {%
    \equal{\zref@extract{#1}{subsection}}%
          {\thesubsection}%
  }{%
    (\zref@extract{#1}{equation})%
  }{%
    (\zref@extract{#1}{subsection}.\zref@extract{#1}{equation})%
  }%
}

\let\oldlabel\label
\renewcommand*{\label}[1]{%
  \oldlabel{#1}%
  \zlabel{#1}%
}

\let\oldltxlabel\ltx@label
\renewcommand*{\ltx@label}[1]{%
  \oldltxlabel{#1}%
  \zlabel{#1}%
}

\let\oldref\ref
\renewcommand*{\ref}[1]{(\oldref{#1})}

\makeatother

\begin{document}

\title{Riemannian Geometry for Shape Analysis and Computational Anatomy}
\author{Martins Bruveris}
\date{November 18, 2017}
\thanks{I would like to thank the organisers, in particular Sergey Kushnarev, for giving me the opportunity to give the course and Jakob M{\o}ller-Andersen for careful proofreading.}

\begin{abstract} Shape analysis and compuational anatomy both make use of sophisticated tools from infinite-dimensional differential manifolds and Riemannian geometry on spaces of functions. While comprehensive references for the mathematical foundations exist, it is sometimes difficult to gain an overview how differential geometry and functional analysis interact in a given problem. This paper aims to provide a roadmap to the unitiated to the world of infinite-dimensional Riemannian manifolds, spaces of mappings and Sobolev metrics: all tools used in computational anatomy and shape analysis.
\end{abstract}

\keywords{Shape analysis; computational anatomy; diffeomorphism group; Riemannian geometry}
\subjclass[2010]{Primary 58B20, 58D15} 

\maketitle

\setcounter{tocdepth}{1}
\tableofcontents

\section*{Introduction}

These lecture notes were written to supplement a course given at the summer school ``Mathematics of Shapes'' in Singapore in July 2016. The aim of the course was to show how the language of differential geometry can be used in shape analysis and computational anatomy. Of course four lectures are not enough to fully do justice to the subject. However four lectures are enough to give an introduction to infinite-dimensional differential geometry, to show how the infinite-dimensional world differs from the finite-dimensinal one and to point the interested reader to more in depth references. Infinite-dimensional differential geometry is treated in detail in \cite{Lang1999,Klingenberg1995,Michor1997} and more information about manifolds of maps can be found in the articles \cite{Bauer2014,Hamilton1982}. For shape analysis one may consult \cite{Srivastava2016} and for computational anatomy \cite{Younes2010}.

\section{Infinite-dimensional manifolds and functional analysis}

The first lecture gave an introduction to manifolds in infinite dimensions assuming familiarity with finite-dimensional differential geometry. Then it discussed the main differences between finite and infinite dimensions: loss of local compactness and no existence theorems for ODEs in Fr\'echet spaces. Finally it disussed Omori's theorem and what it means that the diffeomorphism group of a compact manifold cannot be modelled as a smooth Banach Lie group.

\subsection{Infinite-dimensional manifolds}
A \emph{smooth manifold} modelled on the topological vector space $E$ is a Hausdorff topological space $M$ together with a family of \emph{charts} $(u_\al, U_\al)_{\al \in A}$, such that
\begin{enumeq}
\item
$U_\al \subseteq M$ are open sets, $\bigcup_{\al \in A} U_\al = M$;
\item
$u_\al : U_\al \to u_\al(U_\al) \subseteq E$ are homeomorphisms onto open sets $u_\al(U_\al)$;
\item \label{eq:charts_smooth}
$u_\be \circ u_\al\inv : u_\al(U_\al \cap U_\be) \to u_\be(U_\al \cap U_\be)$ are $C^\infty$-smooth.
\end{enumeq}

In this definition it does not matter, whether $E$ is finite or infinite-dimensional. In fact, if $E$ is finite-dimensional, then $E = \R^n$ for some $n \in \mb N$ and we recover the definition of a finite-dimensional manifold.

\subsection{Choice of a modelling space}

There are several classes of infinite-di\-men\-sio\-nal vector spaces to choose from. With increasing generality our space $E$ can be a
\begin{enumeq}
\item Hilbert space;
\item Banach space;
\item Fr\'echet space;
\item convenient locally convex vector space.
\end{enumeq}
We will assume basic familiarity with Hilbert and Banach spaces. All topological vector spaces are assumed to be Hausdorff. A \emph{Fr\'echet space} is a locally convex topological vector space $X$, whose topology can be induced by a complete, translation-invariant metric, i.e. a metric $d :X\x X \to \R$ such that $d(x+h,y+h)=d(x,y)$. Alternatively a Fr\'echet space can be characterized as a Hausdorff topological space, whose topology may be induced by a countable family of seminorms $\|\cdot\|_n$, i.e., finite intersections of the sets $\{ y \,:\, \| x - y \|_n < \ep\}$ with some $x$, $n$, $\ep$ form a basis of the topology, and the topology is complete with respect to this family.

We will in these lectures ignore for the most part convenient vector spaces; a detailed exposition of manifolds modelled on these spaces can be found in \cite{Michor1997}; we mention only that convenient vector spaces are necessary to model spaces of compactly supported functions on noncompact manifolds.

Each of these classes is more restrictive than the next one. A \emph{Hilbert space} is a vector space with an inner product\footnote{We will ignore questions of completeness in this informal discussion.} $\langle\cdot,\cdot\rangle$. The inner product induces a norm via
\[
\| x\| = \sqrt{\langle x, x \rangle}\,.
\]
If we are just given a norm $\|\cdot\|$, then we have a \emph{Banach space}. A norm can be used to define a distance $d$ via
\[
d(x,y) = \| x - y\|\,
\]
If we have only a distance function, then the space is a \emph{Fr\'echet space}. It is in general not possible to go in the other direction.

\subsection{The Hilbert sphere}
A first example of an infinite-dimensional manifold is the unit sphere in a Hilbert space. Let $E$ be an infinite-dimensional Hilbert space. Then
\[
S = \{ x \in E \,:\, \| x \| = 1 \}
\]
is a smooth manifold. We can construct charts on $S$ in the following way: For $x_0 \in S$, define the subspace $E_{x_0} = \{ y \in E \,:\, \langle y, x_0 \rangle = 0\}$, which is isomorphic to $E$ itself. The chart map is given by
\[
u_{x_0} : x \mapsto x - \langle x, x_0 \rangle x_0\,,
\]
and is defined between the sets $U_{x_0} = \{ x \in S \,:\, \langle x, x_0 \rangle > 0\}$ and $u_{x_0}(U_{x_0}) = \{ y \in E_{x_0} \,:\, \| y \| < 1 \}$. Its inverse is
\[
u_{x_0}\inv : y \mapsto y + \sqrt{1 - \| y \|^2}\, x_0\,.
\]
We will omit the verification that chart changes are smooth maps.

Note that if $E$ is infinite-dimensional, then the sphere is not compact. To see this choose an orthonormal sequence $(e_n)_{n\in \mb N}$. Then $e_n \in S$, but the sequence does not contain a convergent subsequence, because $\| e_n - e_m \| = \sqrt 2$. Hence $S$ cannot be compact.

\subsection{The manifold $\on{Imm}(S^1,\R^d)$}

One of the simplest spaces of functions is the space of smooth, periodic, immersed curves,
\[
\on{Imm}(S^1,\R^d) = \{ c \in C^\infty(S^1,\R^d) \,:\, c'(\th) \neq 0,\, \forall \th \in S^1\}\,.
\]
The modelling space for the manifold is $C^\infty(S^1,\R^d)$, the space of smooth, periodic functions\footnote{A good introduction to the space of smooth functions on the circle can be found in \cite{Garrett2013web}.}. When we talk about manifolds, we have to specify, what topology we mean. On the space $C^\infty(S^1,\R^d)$ of smooth functions we consider the topology of uniform convergence in all derivatives, i.e.,
\[
f_n \to f \Leftrightarrow \lim_{n\to \infty}\| f_n^{(k)} - f^{(k)} \|_\infty = 0\quad \forall k \in \mb N\,,
\]
where $\| f\|_\infty = \sup_{\th \in S^1} |f(\th)|$. A basis of open sets is formed by sets of the form
\[
M(f,\ep,k) = \left\{ g \in C^\infty(S^1,\R^d) \,:\, \| g^{(k)} - f^{(k)} \|_\infty < \ep \right\}\,,
\]
with $\ep \in \R_{>0}$ and $k \in \mb N$.

\begin{proposition*}
$C^\infty(S^1,\R^d)$ is a reflexive, nuclear, separable Fr\'echet space.
\end{proposition*}

See \cite[Sect.~6.2]{Pietsch1972} and \cite[Thm.~4.4.12]{Pietsch1972} for proofs. Having defined a topology, how does the set of immersions sit inside the set of all smooth functions?

\begin{lemma*}
$\on{Imm}(S^1,\R^d)$ is an open subset of $C^\infty(S^1,\R^d)$.
\end{lemma*}

\begin{proof}
Given $f \in \on{Imm}(S^1,\R^d)$, let $\ep = \inf_{\th \in S^1} |f'(\th)|$. We have $\ep > 0$ since $S^1$ is compact. Now consider the neighborhood $M(f,\ep/2,1)$ defined above; for $g \in M(f,\ep/2,1)$ we can estimate
\[
|g'(\th)| \geq |f'(\th)| - \| g' - f'\|_\infty \geq \frac \ep 2 > 0\,,
\]
and thus $f \in M(f,\ep/2,1) \subseteq \on{Imm}(S^1,\R^d)$.
\end{proof}

See \cite[Thm.~2.1.1]{Hirsch1994} for a more general statement about the spaces $\on{Imm}(M,N)$. One can also show\footnote{This is true for $d \geq 2$; for $d=1$ the set $\on{Imm}(S^1,\R^d)$ is empty.} that $\on{Imm}(S^1,\R^d)$ is dense in $C^\infty(S^1,\R^d)$ \cite[Prop.~2.1.0]{Hirsch1994}. Because open subsets of vector spaces are the simplest examples of manifold, $\on{Imm}(S^1,\R^d)$ is a Fr\'echet manifold, modelled on the space $C^\infty(S^1,\R^d)$.

\subsection{The manifold $\on{Imm}_{C^n}(S^1,\R^d)$}
Instead of smooth curves, we could consider curves belonging to some other regularity class. Let $n \geq 1$ and
\[
\on{Imm}_{C^n}(S^1,\R^d) = \{ c \in C^n(S^1,\R^d) \,:\, c'(\th) \neq 0,\, \forall \th \in S^1\}\,,
\]
be the space of $C^n$-immersions. Again we need a topology on the space $C^n(S^1,\R^d)$. In this case
\[
\| f \|_{n,\infty} = \sup_{0 \leq k \leq n} \| f^{(k)}\|_\infty\,,
\]
is a norm making $(C^n(S^1,\R^d), \|\cdot\|_{n,\infty})$ into a Banach space.

\begin{lemma*}
For $n \geq 1$, $\on{Imm}_{C^n}(S^1,\R^d)$ is an open subset of $C^n(S^1,\R^d)$.
\end{lemma*}

\begin{proof}
It is easy to see that the sets $M(f,\ep,k)$ with $0 \leq k \leq n$---after we adapt the definition of $M(f,\ep,k)$ to include all functions $g \in C^n(S^1,\R^d)$ satisfying the inequality $\| g^{(k)} - f^{(k)}\| < \ep$---are open in $C^n(S^1,\R^d)$ and in the proof for $\on{Imm}(S^1,\R^d)$ we only used $k=1$; hence the same proof shows that $\on{Imm}_{C^n}(S^1,\R^d)$ is open in $C^n(S^1,\R^d)$.
\end{proof}

Thus $\on{Imm}_{C^n}(S^1,\R^d)$ is a Banach manifold modelled on the space $C^n(S^1,\R^d)$. There is a connection between the spaces of $C^n$-immersions and the space of smooth immersions. As sets we have
\[
\on{Imm}(S^1,\R^d) = \bigcap_{n \geq 1} \on{Imm}_{C^n}(S^1,\R^d)\,.
\]
However, more is true: we can consider the diagram
\[
C^\infty(S^1,\R^d) \subseteq \dots \subseteq C^n(S^1,\R^d) \subseteq C^{n-1}(S^1,\R^d)
\subseteq \dots \subseteq C^1(S^1,\R^d)\,,
\]
and topologically $C^\infty(S^1,\R^d) = \varprojlim_{n \to \infty} C^n(S^1,\R^d)$ is the projective limit of the spaces $C^n(S^1,\R^d)$ \cite{Garrett2013web}.

\subsection{Calculus in Banach spaces}
Having chosen a modelling space $E$, we look back at the definition of a manifold and see that \eqref{eq:charts_smooth} requires chart changes to be smooth maps. To a considerable extent multivariable calculus generalizes without problems from a finite-dimensional Euclidean space to Banach spaces, but not beyond. 

For example, if $X$, $Y$ are Banach spaces and $f : X \to Y$ a function, we can define the derivative $Df(x)$ of $f$ at $x \in X$ to be the linear map $A \in L(X,Y)$, such that
\[
\lim_{h \to 0} \frac{\|f(x + h) - f(x) - A.h\|_Y}{\| h\|_X} = 0\,.
\]
I want to emphasize in particular two theorems, that are valid in Banach spaces, but fail for Fr\'echet spaces: the existence theorem for ODEs and the inverse function theorem. First, the local existence theorem for ODEs with Lipschitz right hand sides.

\begin{theorem*}
Let $X$ be a Banach space, $U \subseteq X$ an open subset and $F : (a,b) \x U \to X$ a continuous function, that is Lipschitz continuous in the second variable, i.e.,
\[
\| F(t,x) - F(t,y)\|_X \leq C \|x - y \|_X\,,
\]
for some $C > 0$ and all $t \in (a,b)$, $x,y \in U$. Then, given $(t_0, x_0) \in (a,b) \x U$, there exists $x : (t_0-\ep, t_0+\ep) \to X$, such that
\[
\p_t x(t) = F(t, x(t))\,,\quad x(t_0) = x_0\,.
\]
\end{theorem*}

In fact more can be said: the solution is as regular as the right hand side; if the right hand side depends smoothly on some parameters, then so does the solution; furthermore, one can estimate the length of the interval of existence. One can also get by with less regularity of $F(t,x)$ in the $t$-variable. This is used in the LDDMM framework~\cite{Younes2010}, where vector fields $F(t,x)$ are assumed to be only integrable in the $t$-variable but not necessarily continuous.

For $x \in X$, denote by $B_r(x) = \{ y \in X \,:\, \| y - x\|_X < r \}$ the open $r$-ball. The following is a version of the inverse function theorem.

\begin{theorem*}
Let $X$, $Y$ be Banach spaces, $U \subseteq X$ open, $f \in C^1(U,Y)$ and $Df(x_0)$ invertible for $x_0 \in U$. Then there exists $r > 0$ such that $f(B_r(x_0))$ is open in $Y$ and $f: B_r(x_0) \to f(B_r(x_0))$ is a diffeomorphism.
\end{theorem*}

Both of these theorems are not valid in Fr\'echet spaces with counterexamples given in \ref{ss:frechet_counterexamples}. But even in Banach spaces life is not as easy as it was in finite dimensions. Two things are lost: local compactness and uniqueness of the topology. In fact the only locally compact vector spaces are finite-dimensional ones.

\begin{theorem*}[{\cite[Thm.~1.22]{Rudin1991}}]
Let $X$ be a topological vector space. If $X$ has an open set, whose closure is compact, then $X$ is finite-dimensional.
\end{theorem*}

In finite dimensions we do not have to choose, which topology to consider: there is only one $n$-dimensional vector space.

\begin{theorem*}[{\cite[Thm.~1.21]{Rudin1991}}]
Let $X$ be a topological vector space. If $\dim X = n$, then $X$ is homeomorphic to $\R^n$.
\end{theorem*}

This is lost in infinite dimensions. There is some limited variant of the uniqueness for Banach spaces.

\begin{theorem*}[\cite{Tao2016web}]
Let $(X, \mc F)$ be a topological vector space. Then, up to equivalence of norms, there is at most one norm $\| \cdot\|$ one can place on $X$, such that $(X, \|\cdot\|)$ is a Banach space whose topology is at least as strong as $\mc F$. In particular, there is at most one topology stronger than $\mc F$ that comes from a Banach space norm.
\end{theorem*}

\subsection{Counterexamples in Fr\'echet spaces}
\label{ss:frechet_counterexamples}

For an example, that differential equations may not have solutions in Fr\'echet spaces, consider $C(\R)$, the space of continuous functions, with the compact open topology \cite[Sect.~2.1]{Hirsch1994}. A basis for the compact open topology consists of sets
\[
M(K,V) = \{ f \in C(\R) \,:\, f(K) \subseteq V \}\,,
\]
where $K\subseteq \R$ is compact and $V \subseteq \R$ is open. The differential equation
\[
\p_t f = f^2\,,\quad f(0,x) = x\,,
\]
has a smooth right hand side, but admits no solution in $C(\R)$: if we look at the pointwise solution, we have
\[
f(t,x) = \frac{x}{1-tx}\,,
\]
provided $tx < 1$. Hence for no $t \neq 0$ do we obtain a function, defined on all of $\R$.

For an example, that the inverse function theorem fails for Fr\'echet spaces,  consider the map
\[
F : C(\R) \to C(\R)\,,\quad f \mapsto e^f\,,
\]
where $C(\R)$ carries the compact open topology. Its derivative is $DF(f).h = e^f.h$, which is invertible everywhere. The image of $F$ consists of everywhere positive functions,
\[
F(C(\R)) = \{ f \in C(\R) \,:\, f > 0 \}\,,
\]
but this set is not open in the compact open topology, because the topology is not strong enough to control the behaviour towards infinity.

\subsection{Banach and Fr\'echet manifolds}

Why are Banach or even Hilbert manifolds not enough? One of the important objects in computational anatomy and shape analysis is the diffeomorphism group
\[
\on{Diff}(M) = \{ \ph \in C^\infty(M,M) \,:\, \ph \text{ bijective},\, \ph\inv \in C^\infty(M,M) \}\,,
\]
of a compact manifold. We will see later that $\on{Diff}(M)$ is a smooth Fr\'echet--Lie group. What about a Banach manifold version of the diffeomorphism group? If $n \geq 1$, then one can consider
\[
\on{Diff}_{C^n}(M) = \{ \ph \in C^n(M,M) \,:\, \ph \text{ bijective},\, \ph\inv \in C^n(M,M)\}\,,
\]
the group of $C^n$-diffeomorphisms. The space $\on{Diff}_{C^n}(M)$ is a Banach manifold and a topological group, but not a Lie group. What went wrong? The group operations are continuous, but not differentiable. Fix $\ph \in \on{Diff}_{C^n}(M)$ and consider the map
\[
L_\ph: \on{Diff}_{C^n}(M) \to \on{Diff}_{C^n}(M)\,,\quad \ps \mapsto \ph \circ \ps\,;
\]
its derivative should be
\[
T_\ps L_\ph.h = (D\ph \circ \ps).h\,,
\]
with $T_\ps L_\ph$ denoting the derivative of the map $L_\ps$ and $D\ph$ denotes the derivative of the diffeomorphism $\ph$; the former is a map between infinite-dimensional manifolds, while the latter maps $M$ to itself. To see this, consider a one-parameter variation $\ps(t,x)$, such that $\ps(0,x) = \ps(x)$ and $\p_t \ps(t,x)|_{t=0} = h(x)$, and compute
\[
\p_t \ph(\ps(t,x))|_{t=0} = D\ph(\ps(x)).h(x)\,.
\]
We see that in general $T_\ps L_\ph.h$ lies only in $C^{n-1}$. However, if composition were to be a differentiable operator, $T_\ph L_\ps$ would have to map into $C^n$-functions.

There seems to be a trade off involved: we can consider smooth functions, in which case the diffeomorphism group is a Lie group, but can be modelled only on a Fr\'echet space; or we look at functions with finite regularity, but then composition ceases to be differentiable. There is a theorem by Omori \cite{Omori1978} stating that this choice cannot be avoided.

\begin{theorem*}[Omori, 1978 \cite{Omori1978}]
If a connected Banach--Lie group $G$ acts effectively, transitively and smoothly on a compact manifold, then $G$ must be a finite-dimensional Lie group.
\end{theorem*}

A smooth action of a Lie group $G$ on an manifold $M$ is a smooth map $G \x M \to M$, written as $(g,x) \mapsto g.x$, satisfying the identities $e.x = x$ and $g.(h.x) = (gh).x$, for all $g,h \in G$ and $x \in M$ with $e \in G$ the identity element \cite[Sect.~6]{Michor2008b}. The action is called \emph{transitive}, if for any two points $x,y \in M$ there exists $g \in G$ with $g.x = y$; the action is called \emph{effective}, if
\[
g.x = h.x \text{ for all } x \in M \Rightarrow g = h\,.
\]
In other words, an effective action allows us to distinguish group elements based on their action on the space. 

The action of the diffeomorphism group on the base manifold is given by $\ph.x = \ph(x)$; it is clearly effective, since $\ph(x) = \ps(x)$ for all $x \in M$ implies $\ph = \ps$ as functions. The diffeomorphism group also acts transitively \cite[(43.20)]{Michor1997}.

Therefore Omori's theorem requires us to make a choice: either our diffeomorphism group is not a Banach manifold or it is not a smoothly acting Lie group, i.e., the group operations or the action on the manifold are not smooth. Choosing to work with the group $\on{Diff}(M)$ of smooth diffeomorphisms leads to the Fr\'echet manifold setting, where $\on{Diff}(M)$ is a Lie group; it is easier to do geometry, since more operations are differentiable, but establishing analytic results is more challenging. The other choice is a group like $\on{Diff}_{C^n}(M)$ of diffeomorphisms with finitely many derivatives. This group is a Banach manifold and hence one has multiple tools available to prove existence results; however, because $\on{Diff}_{C^n}(M)$ is not a Lie group, it is a less rich geometric setting. One can use the intuition and the language of differential geometry, but not necessarily its tools.

\section{Riemannian geometry in infinite dimensions}

Most examples of spaces of maps discussed in the first lecture were open subsets of vector spaces. For example, $\on{Imm}(S^1,\R^d)$ is an open subset of $C^\infty(S^1,\R^d)$ and as such its manifold structure is trivial. The second lecture begun by discussing how to define a manifold structure on nonlinear spaces of functions such as $C^\infty(M,N)$. It then proceeded to consider Riemannian metrics on infinite-dimensional manifolds with special emphasis on weak Riemannian metrics. These are Riemannian metrics on Banach and Fr\'echet manifolds.

\subsection{The manifold $C^\infty(M,N)$}

We assume that $M$ is a compact (hence finite-dimensional) manifold without boundary, while $N$ can be a noncompact (with some mild restrictions even an infinite-dimensional) manifold, also without boundary. How do we model $C^\infty(M,N)$ as an infinite-dimensinal manifold? A local deformation $h$ of a function $f \in C^\infty(M,N)$ is a vector field in $N$ along $M$,
\[
\begin{aligned} \xymatrix{
& TN \ar[d]^{\pi_N} \\
M \ar[r]^f \ar[ur]^h & N
} \end{aligned}\qquad
h(x) \in T_{f(x)}N\,.
\]
The collection of all these deformations is
\[
\Ga(f^\ast TN) = \left\{ h \in C^\infty(M,TN) \,:\, \pi_N \circ h = f \right\}\,,
\]
the space of sections of the pullback bundle $f^\ast TN$. This is a linear space.

Choose a Riemannian metric $(N, \bar g)$ on $N$. The metric gives rise to an exponential map
\[
\pi_N \x \on{exp} : TN \supseteq U \to N \x N\,,\quad
h_x \mapsto (y, \on{exp}_y(h_y))\,,
\]
where $h_y \in T_y N$ and $\on{exp}_y$ is the exponential map of $\bar g$ at $y \in N$. It is defined on a neighborhood of the zero section, and if $U$ is suitably small, it is a diffeomorphism onto its image. Denote the image by $V = \pi_N \x \on{exp}(U)$. Then any function $g \in C^\infty(M,N)$, that is close enough to $f$ in the sense that $(f(x), g(x)) \in V$ for all $x \in M$, can be represented by
\[
u_f(g): x \mapsto \left(\pi_N \x \on{exp}\right)\inv(f(x), g(x))\,,\quad
u_f(g) \in \Ga(f^\ast TN)\,.
\]
This means that we have found around each function $f$ an open neighborhood
\[
\mc V_f = \{ g \in C^\infty(M,N) \,:\, (f(x), g(x)) \in V \text{ for all }x \in M\}\,,
\]
and $\mc V_f$ can be mapped bijectively onto the open subset
\[
\mc U_f = \{ h \in \Ga(f^\ast TN) \,:\, h(x) \in U \text{ for all } x\in M \}\,,
\]
of the vector space $\Ga(f^\ast TN)$.

Several things have been left unsaid: one has to check that this map is indeed continuous and that its inverse is continuous as well; one also has to check that the chart change maps $h \mapsto u_f \circ u_g\inv(h)$ are smooth as functions of the vector fields $h$.

One can use the same method to construct charts on the spaces $C^n(M,N)$ of functions with finitely many derivatives. The main problem is that if $f : M \to N$ is not smooth, then the pullback bundle $f^\ast TN$ is not a smooth (finite-dimensional) manifold any more. To overcome this difficulty we can use that $C^\infty(M,N)$ is dense in $C^n(M,N)$ and construct charts around all smooth $f$. We then have to check, that these charts indeed cover all of $C^n(M,N)$\footnote{Density of $C^\infty(M,N)$ alone is not sufficient. For example $\mb Q$ is dense in $\R$, but we can choose a small interval around each rational number in such a way that all the intervals still miss $\sqrt{2}$.}.

\subsection{Strong and weak Riemannian manifolds}

Let $M$ be a manifold modelled on a vector space $E$. A \emph{weak Riemannian metric} $G$ is a smooth map
\[
G : TM \x_M TM \to \R\,,
\]
satisfying
\begin{enumeq}
\item
$G_x(\cdot,\cdot)$ is bilinear for all $x \in M$;
\item
$G_x(h,h) \geq 0$ for all $h \in T_x M$ with equality only for $h=0$.
\end{enumeq}
This implies that the associated map
\[
\check G : TM \to T^\ast M\,,\quad \left\langle \check G_x(h), k \right\rangle = G_x(h,k)\,,
\]
is injective. In finite dimensions it would follow by counting dimensions that $\check G$ is bijective and hence an isomorphism. In infinite dimensions this is not longer the case.

\begin{example*}
Consider the space of smooth curves $\on{Imm}(S^1,\R^d)$ with the $\on{Diff}(S^1$-invariant $L^2$-metric
\[
G_c(h,k) = \int_{S^1} \langle h(\th), k(\th) \rangle |c'| \ud \th\,.
\]
Then $\check G_c(h) = h.|c'|$ and the image of $T_c \on{Imm}(S^1,\R^d) = C^\infty(S^1,\R^d)$ under $\check G_c$ is again $C^\infty(S^1,\R^d)$, while the dual space $T^\ast_c \on{Imm}(S^1,\R^d) = \mc D'(S^1)^d$ is the space of $\R^d$-valued distributions.
\end{example*}

This means that in infinite dimensions we have to distinguish between two different notions of Riemannian metrics. A \emph{strong Riemannian metric} is required to additionally satisfy
\begin{enumeq}
\item
The topology of the inner product space $(T_xM, G_x(\cdot,\cdot))$ coincides with the topology $T_x M$ inherits from the manifold $M$.
\end{enumeq}
A strong Riemannian metric implies that $T_x M$ and hence the modelling space of $M$ is a Hilbert space. See \cite{Klingenberg1995,Lang1999} for the theory of strong Riemannian manifolds.

\begin{example*}
Let $H$ be a Hilbert space. Then the Hilbert sphere
\[
S = \{ x \in H \,:\, \| x \|=1 \}\,,
\]
with the induced Riemannian metric $G_x(h,k) = \langle h, k \rangle$ for $h,k \in T_x S = \{ h \in H \,:\, \langle h,x \rangle = 0\}$ is a strong Riemannian manifold.
\end{example*}

Why do we consider weak Riemannian manifolds? There are two reasons: the only strong Riemannian manifolds are Hilbert manifolds; when we want to work with the space of smooth functions, any Riemannian metric on it will be a weak one; the other reason is that some Riemannian metrics, that are important in applications (the $L^2$-metric on the diffeomorphism group for example) cannot be made into strong Riemannian metrics.

\subsection{Levi-Civita covariant derivative}

After defining the Riemannian metric, one of the next objects to consider is the covariant derivative. Let $(M,G)$ be a Riemannian manifold, modelled on $E$ and $X,Y,Z$ vector fields on $M$. Assume that $M \subseteq E$ is open or that we are in a chart for $M$. Then the Levi-Civita covariant derivative is given by
\[
\nabla_X Y(x) = DY(x).X(x) + \Ga(x)(X(x),Y(x))\,,
\]
where $\Ga : M \to L(E,E;E)$ are the \emph{Christoffel symbols} of $G$,
\begin{equation}
\label{eq:define_christoffel}
2 G(\Ga(X,Y),Z) = D_{\cdot,X}G_\cdot(Y,Z) + D_{\cdot,Y}G_\cdot(Z,X) - D_{\cdot,Z}G_\cdot(X,Y)\,.
\end{equation}
The important part to note is that in the definition of $\Ga$ one uses the inverse of the metric. In fact we don't need to be able to always invert it, but we need to now that the right hand side of \eqref{eq:define_christoffel} lies in the image $\check G(TM)$ of the tangent bundle under $\check G$. For strong Riemannian metrics this is the case, but not necessarily for weak ones.

\subsection{Example. The $L^2$-metric}

Consider for example the space of $C^1$-curves
\[
\on{Imm}_{C^1}(S^1,\R^d) = \{ c \in C^1(S^1,\R^d) \,:\, c'(\th) \neq 0,\, \forall \th \in S^1\}\,,
\]
with the $L^2$-metric
\[
G_c(h,k) = \int_{S^1} \langle h(\th), k(\th) \rangle |c'(\th)| \ud \th\,.
\]
Then one can calculate that
\[
D_{c,l}G_\cdot(h,k) = \int_{S^1} \langle h, k\rangle \langle l', c'\rangle \frac{1}{|c'|} \ud \th\,,
\]
and we see that the right hand side of \eqref{eq:define_christoffel}, which is $2G(\Ga(h,k),l)$ in this notation, involves derivatives of $l$, while the left hand side does not. While this is not a complete proof, it shows the idea, why the $L^2$-metric on the space of curves with a finite number of derivatives does not have a covariant derivative.

\subsection{The geodesic equation}
The geodesic equation plays an important role in both shape analysis and computational anatomy. Informally it describes least-energy deformations of shapes or optimal paths of transformations. From a mathematical point we can write it in a coordinate-free way as
\[
\nabla_{\dot c}{\dot c} = 0\,,
\]
and in charts it becomes
\[
\ddot c + \Ga(c)(\dot c, \dot c) = 0\,.
\]
From this it seems clear that the geodesic equation needs the covariant derivative or equivalently the Christoffel symbols to be written down. A more concise way to say that a metric does not have a covariant derivative would be to say that the geodesic equation for the metric does not exist. Now, how can an equation fail to exist? The geodesic equation corresponds to the Euler--Lagrange equation of the energy function
\[
E(c) = \frac 12 \int_0^1 G_c(\dot c, \dot c) \ud t\,,
\]
and a geodesic is a critical point of the energy functional, restricted to paths with fixed endpoints. In a coordinate chart we can differentiate the energy functional
\[
D_{c,h}E = \int_0^1 G_c(\dot c, \dot h) + \frac 12 D_{c,h}G_\cdot(\dot c, \dot c) \ud t\,.
\]
The steps until now can be done with any metric. What cannot always be done is to isolate $h$ in this expression to obtain something of the form $\int_0^1 G_c(\dots, h) \ud t$,
where the ellipsis would contain the geodesic equation.

\subsection{The geodesic distance}

Let $(M,G)$ be a (weak) Riemannian manifold and assume $M$ is connected. For $x,y \in M$ we can define the \emph{geodesic distance} between them as in finite dimensions,
\[
\on{dist}(x,y) = \inf_{\substack{c(0) = x\\ c(1)=y}} \int_0^1 \sqrt{G_c(\dot c, \dot c)} \ud t\,,
\]
where the infimum is taken over all smooth paths or equivalently all piecewise $C^1$-paths. Then $\on{dist}$ has the following properties:
\begin{enumeq}
\item
$\on{dist}(x,y) \geq 0$ for $x,y \in M$;
\item
$\on{dist}(x,y) = \on{dist}(y,x)$;
\item
$\on{dist}(x,z) \leq \on{dist}(x,y) + \on{dist}(y,z)$.
\end{enumeq}
What is missing from the list of properties?
\begin{enumeq}
\item
$\on{dist}(x,y) \neq 0$ for $x \neq y$.
\end{enumeq}
We call this last property \emph{point-separating}\footnote{There is a slight difference between a point-separating and a nonvanishing geodesic distance. We say that $\on{dist}$ is \emph{nonvanishing}, if there exist $x,y \in M$, such that $\on{dist}(x,y) \neq 0$. It follows that a point-separating distance is nonvanishing, but in general not the other way around.}. It may fail to hold for weak Riemannian metrics. This is a purely infinite-dimensional phenomenon; in fact it only happens for weak Riemannian metrics and there are explicit examples of this \cite{Michor2005,Bauer2014}.

Note that vanishing of the geodesic distance does not mean that the metric itself is degenerate. In fact, if $x \neq y$ and $c : [0,1] \to M$ is a path with $c(0) = x$ and $c(1) = y$, then we have
\[
\on{Len}(c) = \int_0^1 \sqrt{G_c(\dot c, \dot c)} \ud t > 0\,,
\]
with a strict inequality and if $\on{Len}(c)= 0$ then $c$ must be the constant path. Thus $\on{dist}(x,y) = 0$ arises, because there might exist a family of paths with positive, yet arbitrary small, length connecting the given points.

What is the \emph{topology induced by the geodesic distance}? In finite dimensions and in fact for strong manifolds we have the following theorem.

\begin{theorem*}[{\cite[Thm.~1.9.5]{Klingenberg1995}}]
Let $(M,G)$ be a strong Riemannian manifold. Then $\on{dist}$ is point-separating and the topology induced by $(M, \on{dist})$ coincides with the manifold topology.
\end{theorem*}

\section{Complete Riemannian manifolds and Hopf--Rinow}

The third lecture discussed completeness properties and how far the theorem of Hopf--Rinow can be generalized to infinite-dimensional manifolds. The second part studied the group of diffeomorphisms of Sobolev regularity in more detail. This group is used in computational anatomy to model anatomical deformations and it is of interest to establish completeness results for Sobolev metrics on this group.

\subsection{Completeness properties}
For a Riemannian manifold $(M,G)$ completeness can mean several things.
\begin{enumeq}
\item
\label{eq:metric_compleness}
$M$ is \emph{metrically complete}, meaning that $(M,\on{dist})$ is a complete metric space; i.e., all Cauchy sequences with respect to $\on{dist}$ converge.
\item
\label{eq:geodesic_completeness}
$M$ is \emph{geodesically complete}, meaning that every geodesic can be continued for all time.
\item
\label{eq:ex_min_geodesics}
Between any two points on $M$ (in the same connected component), there \emph{exists a length minimizing geodesic}.
\end{enumeq}
In finite dimensions the \emph{theorem of Hopf--Rinow} states that (1) and (2) are equivalent and imply (3). The only implication one has in infinite dimensions is that on a strong Riemannian manifold metric completeness implies geodesic compleness.
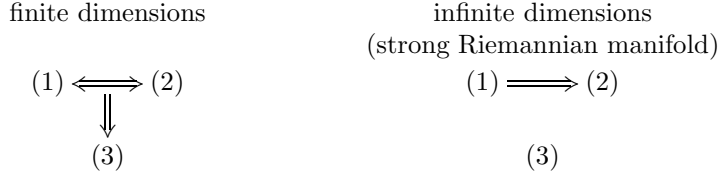
\[
\begin{array}{c@{\qquad\qquad\qquad}c}
\text{finite dimensions} & \text{infinite dimensions} \\
& \text{(strong Riemannian manifold)} \\
\xymatrix@R=1pc@C=0.3pc{
 \mathrm{(1)} \ar@{<=>}[rr] & \ar@{=>}[d] & \mathrm{(2)} \\ 
& \mathrm{(3)} } &
\xymatrix@R=1pc@C=0.3pc{
 \mathrm{(1)} \ar@{=>}[rr] & & \mathrm{(2)} \\ 
& \mathrm{(3)} }
\end{array}
\]
In general one cannot expect more and there are explicit counterexamples. Atkin \cite{Atkin1975} found an example of a metrically and geodesically complete manifold with two points that cannot be joined by \emph{any} geodesic (not just a minimizing one), showing that
\[
\text{(1) \& (2) $\nRightarrow$ (3)}\,;
\]
Two decades later Atkin \cite{Atkin1997} showed also that one can find a geodesically complete manifold satisfying (3), which is not metrically complete, thus showing
\[
\text{(2) \& (3) $\nRightarrow$ (1)}\,.
\]
A simple example showing that metric and geodesic completeness do not imply existence of minimizing geodesics is Grossman's ellipsoid.

\subsection{Grossman's ellipsoid}
The presentation follows \cite[Sect.~VIII.6]{Lang1999}. Consider a separable Hilbert space $E$ with an orthonormal basis $(e_n)_{n \in \mb N}$ and define the sequence $(a_n)_{n \in \mb N}$ by $a_0 = 1$ and $a_n = 1 + 2^{-n}$ for $n \geq 1$. Consider the ellipsoid
\[
M = \left\{ \sum_{n \in \mb N} x_n e_n \in E \,:\, \sum_{n \in \mb N} \frac{x_n^2}{a_n^2}= 1 \right\}\,.
\]
We can view $M = F(S)$ as the image of the unit sphere $S = \{ x \in X \,:\, \| x \| = 1\}$ under the transformation
\[
F : X \to X\,,\quad
\sum_{n \in \mb N} x_n e_n \mapsto \sum_{n \in \mb N} a_n x_n e_n\,.
\]

Consider a path $c$ in $S$ joining the two points $e_0$ and $-e_0$. Then $Fc$ is a path in $M$ joining $e_0$ and $-e_0$ and every path in $M$ can be written in such a way, because $F$ is invertible. We claim that $\on{dist}(e_0,-e_0) = \pi$, but that there exists no path realizing this distance. The lengths of a path $c(t) = \sum_{n \in \mb N} c_n(t) e_n$ in $S$ and of $Fc$ in $M$ are
\begin{align*}
\on{Len}(c) &= \int_0^1 \sqrt{ \sum_{n \in \mb N} \dot c_n(t)^2 } \ud t &
\on{Len}(Fc) &= \int_0^1 \sqrt{ \sum_{n \in \mb N} a_n^2 \dot c_n(t)^2 } \ud t
\end{align*}
We certainly have
\[
\pi \leq \on{Len}(c) \leq \on{Len}(Fc)\,,
\]
since $\| \dot c\| \leq \| F\dot c\|$. In fact by looking at the sequence $(a_n)_{n\in \mb N}$ we see that $\on{Len}(c) = \on{Len}(Fc)$ for a curve $c(t) = \sum_{n \in \mb N} c_n(t) e_n$, if and only if $\dot c_n(t) = 0$ for $n \geq 1$. But the only curve in $S$ starting at $e_0$ that satisfies this is the constant curve. Thus we have for $c$ in $S$ joining $e_0$ and $-e_0$ the  strict inequality
\[
\pi \leq \on{Len}(c) < \on{Len}(Fc)\,,
\]
showing that $\on{dist}_M(e_0,-e_0) \geq \pi$ and $\on{Len}(Fc) > \pi$ for any curve $Fc$ in $M$ joining them. However, if we let $c$ be the half great circle joining the two points in the $(e_0, e_n)$-plane, then
\[
\on{Len}(Fc) \leq \left(1 + 2^{-n}\right) \pi \to \pi = \on{dist}_M(e_0, -e_0)\,.
\]

\subsection{Sobolev spaces on $\R^n$}
The Sobolev spaces $H^q(\R^d)$ with $q \in \R_{\geq 0}$ can be defined in terms of the Fourier transform
\[ 
\mc F f(\xi) = (2\pi)^{-n/2} \int_{\R^n} e^{-i \langle  x,\xi\rangle} f(x) \ud x\,,
\]
and consist of $L^2$-integrable functions $f$ with the property that $(1+|\xi|^2)^{q/2} \mc F f$ is $L^2$-integrable as well. The same definition can also be used when $q < 0$, but then we have to consider distributions $f$, such that $(1+|\xi|^2)^{q/2} \mc F f$ is an $L^2$-integrable function. An inner product on $H^q(\R^d)$ is given by
\begin{equation}
\label{eq:sobolev_inner_prod_fourier}
\langle f, g \rangle_{H^q} = \mathfrak{Re}\int_{\R^d} (1 + |\xi|^2)^q \mc F f(\xi) \overline{\mc F g(\xi)} \ud \xi\,.
\end{equation}

If $q \in \mb N$, the Sobolev space $H^q(\R^d)$ consists of $L^2$-integrable functions $f:\R^d\to \R$ with the property that all distributional derivatives $\p^\al f$ up to order $|\al|\leq q$ are $L^2$-integrable as well. An inner product, that is equivalent but not equal to the above is
\begin{equation}
\label{eq:sobolev_inner_prod_integer}
\langle f, g \rangle_{H^q} = \int_{\R^d} f(x)g(x) + \sum_{|\al|=q} \p^\al f(x) \p^\al g(x) \ud x\,.
\end{equation}

Sobolev spaces satisfy the following embedding property.
\begin{lemma*}
If $q > d/2 + k$, then $H^q(\R^d) \hookrightarrow C^k_0(\R^d)$.
\end{lemma*}

In the above $X \hookrightarrow Y$ means that $X$ is continuously embedded into $Y$ and $C^k_0(\R^d)$ denotes the Banach space $k$-times continuously differentiable functions, that together with their derivatives vanish at infinity. When $q > d/2$, Sobolev spaces also form an algebra.

\begin{lemma*}
If $q > d/2$ and $0 \leq r \leq q$. Then pointwise multiplication can be extended to a bounded bilinear map
\[
H^q(\R^d) \x H^r(\R^d) \to H^r(\R^d)\,,\quad (f,g) \mapsto f \cdot g\,.
\]
\end{lemma*}

There are several equivalent ways to define Sobolev spaces and these definitions all lead to the same set of functions with the same topology. However, the inner products, while equivalent, are not the same. For example for $q \in \mb N$, \eqref{eq:sobolev_inner_prod_fourier} and \eqref{eq:sobolev_inner_prod_integer} define two equivalent, but different inner products on $H^q(\R^d)$.

For the theory of Sobolev spaces one can consult one of the many books on the subject, e.g. \cite{Adams2003}.

\subsection{The diffeomorphism group $\on{Diff}_{H^q}(\R^d)$}
Denote by $\on{Diff}_{C^1}(\R^d)$ the space of $C^1$-diffeomorphisms of $\R^d$, i.e.,
\[
\on{Diff}_{C^1}(\R^d) = \{ \ph \in C^1(\R^d,\R^d) \,:\, \ph \text{ bijective, } \ph\inv \in C^1(\R^d,\R^d) \}\,.
\] 
For $q > d/2+1$ and $q \in \R$ there are three equivalent ways to define the group $\on{Diff}_{H^q}(\R^d)$ of Sobolev diffeomorphisms:
\begin{align*}
\on{Diff}_{H^q}(\R^d) &= \{ \ph \in \on{Id} + H^q(\R^d,\R^d) \,:\, \ph \text{ bijective, }
\ph\inv \in \on{Id} + H^q(\R^d,\R^d) \} \\
&= \{ \ph \in \on{Id} + H^q(\R^d,\R^d) \,:\, 
\ph \in \on{Diff}_{C^1}(\R^d) \} \\
&= \{ \ph \in \on{Id} + H^q(\R^d,\R^d) \,:\, 
\det D\ph(x) > 0,\, \forall x \in \R^d \}\,.
\end{align*}
If we denote the three sets on the right by $A_1$, $A_2$ and $A_3$, then it is not difficult to see the inclusions $A_1 \subseteq A_2 \subseteq A_3$. The equivalence $A_1 = A_2$ has first been shown in \cite[Sect. 3]{Ebin1970b} for the diffeomorphism group of a compact manifold; a proof for $\on{Diff}_{H^q}(\R^d)$ can be found in \cite{Inci2013}. Regarding the inclusion $A_3 \subseteq A_2$, it is shown in \cite[Cor. 4.3]{Palais1959} that if $\ph \in C^1$ with $\det D\ph(x) > 0$ and $\lim_{|x |\to \infty} | \ph(x)| = \infty$, then $\ph$ is a $C^1$-diffeomorphism.

It follows from the Sobolev embedding theorem, that $\on{Diff}_{H^q}(\R^d) - \on{Id}$ is an open subset of $H^q(\R^d,\R^d)$ and thus a Hilbert manifold. Since each $\ph \in \on{Diff}_{H^q}(\R^d)$ has to decay to the identity for $|x|\to \infty$, it follows that $\ph$ is orientation preserving.

\begin{proposition*}[{\cite[Thm.~1.1]{Inci2013}}]
Let $q > d/2 + 1$. Then $\on{Diff}_{H^q}(\R^d)$ is a smooth Hilbert manifold and a topological group.
\end{proposition*}

We have the following result concerning the regularity of the composition map.

\begin{proposition*}[{\cite[Thm.~1.1]{Inci2013}}]
Let $q > d/2+1$ and $k \in \mb N$. Then composition
\[
\on{Diff}_{H^{q+k}}(\R^d) \x \on{Diff}_{H^{q}}(\R^d) \to \on{Diff}_{H^{q}}(\R^d)\,,
\quad (\ph, \ps) \mapsto \ph \circ \ps\,,
\]
and the inverse map
\[
\on{Diff}_{H^{q+k}}(\R^d) \to \on{Diff}_{H^{q}}(\R^d)\,,
\quad \ph \mapsto \ph\inv\,,
\]
are $C^k$-maps.
\end{proposition*}

This proposition means that we have to trade regularitiy of diffeomorphisms to obtain regularity of the composition map, i.e., if $\ph$ is of class $H^{q+k}$, then the composition into $H^q$-diffeomorphisms will be of class $C^k$.

We can also look at the group
\[
\on{Diff}_{H^\infty}(\R^d) = \bigcap_{q > d/2+1} \on{Diff}_{H^q}(\R^d)\,;
\]
it consists of smooth diffeomorphisms that, together with all derivatives, decay towards infinity like $L^2$-functions. It is a smooth, regular, Fr\'echet--Lie group \cite{Michor2013b}; its Lie algebra is $\mf X_{H^\infty}(\R^d) = \bigcap_{q > d/2+1} H^{q}(\R^d)$.

\subsection{Connection to LDDMM}
One beautiful property of the Sobolev diffeomorphism group is that it coincides with the group
\[
\mc G_{H^q(\R^d,\R^d)} = \left\{ \ph(1) \,:\, \ph(t)\text{ is the flow of some }u \in L^1([0,1],H^q(\R^d,\R^d)) \right\}\,,
\]
which is used in the LDDMM framework. We have
\begin{proposition*}[{\cite[Thm.~8.3]{Bruveris2017b}}]
Let $q > d/2 +1$. Then
\[
\mc G_{H^q(\R^d,\R^d)} = \on{Diff}_{H^q}(\R^d)_0\,,
\]
where the space on the right is the connected component of $\on{Id}$.
\end{proposition*}

\begin{proof}
Let $U$ be a convex neighborhood around $\on{Id}$ in $\on{Diff}_{H^q}(\R^d)$. Then every $\ps \in U$ can be reached from $\on{Id}$ via the smooth path $\ph(t) = (1-t) \on{Id} + t \ps$. Since $\ph(t)$ is the flow of the associated vector field $u(t) = \p_t \ph(t) \circ \ph(t)\inv$ and $u \in C([0,1], H^q)$, it follows that $\ps \in \mc G_{H^q}$. Thus $U \subseteq \mc G_{H^q}$ and since $\mc G_{H^q}$ is a group, the same holds also for the whole connected component containing $U$. This shows the inclusion $\on{Diff}_{H^q}(\R^d)_0 \subseteq \mc G_{H^q}$.

For the inclusion $\mc G_{H^q} \subseteq \on{Diff}_{H^q}(\R^d)$ we have to show that given a vector field $u \in L^1([0,1], H^q(\R^d,\R^d))$ the flow defined by $\p_t \ph(t) = u(t) \circ \ph(t)$ is a curve in $\on{Diff}_{H^q}(\R^d)$. This is the content of \cite[Thm.~4.4]{Bruveris2017b}.
\end{proof}

\section{Riemannian metrics induced by the diffeomorphism group}

The last lecture considered the action of the diffeomorphism group $\on{Diff}(\R^d)$ on the space of embeddings $\on{Emb}(M,\R^d)$ given by composition, $(\ph, q) \mapsto \ph \circ q$. Given a Riemannian metric on $\on{Diff}(\R^d)$ as is the case in the LDDMM framework in computational anatomy this action can be used to induce a Riemannian metric on $\on{Emb}(M,\R^d)$ such that for a given $q_0 \in \on{Emb}(M,\R^d)$ the projection $\ph_{q_0}(\ph) = \ph \circ q_0$ is a Riemannian submersion. This lecture looks at this construction and properties of the induced Riemannian metric. Some more recent analytical results can be found in \cite{Bruveris2017b_preprint}.

\subsection{The space $\on{Emb}(M,\R^d)$}
Let $M$ be a compact manifold without boundary. We denote the space of embeddings of $M$ into $\R^d$ by
\[
\on{Emb}(M,\R^d) = \{ q \in C^\infty(M,\R^d) \,:\, q\text{ is an embedding}\}\,;
\]
to be more precise an embedding $q$ is an immersion ($T_x q$ is injective for all $x \in M$) and a homeomorphism onto its image. It is an open subset of the space of immersions, $\on{Imm}(M,\R^d)$, and thus also of $C^\infty(M,\R^d)$; hence it is a Fr\'echet manifold.

The shape space of embeddings is
\[
B_e(M,\R^d) := \on{Emb}(M,\R^d) / \on{Diff}(M)\,.
\]
It can be identified with the set of all embedded submanifolds of $\R^d$, that are diffeomorphic to $M$. Regarding its manifold structure we have the following theorem.

\begin{theorem*}[{\cite[Thm.~1.5]{Michor1991}}]
The quotient space $B_e(M,\R^d)$ is a smooth Hausdorff manifold and the projection
\[
\pi : \on{Emb}(M,\R^d) \to B_e(M,\R^d)
\]
is a smooth principal fibration with $\on{Diff}(M)$ as structure group. 
\end{theorem*}

When $\dim M = d-1$ and $M$ is orientable we can define a chart around $\pi(q) \in B_e(M,\R^d)$ with $q \in \on{Emb}(M,\R^d)$ by
\[
\pi \circ \ps_q : C^\infty(M,(-\ep,\ep)) \to B_e(M,\R^d)\,,
\]
with $\ep$ sufficiently small, where $\ps_q : C^\infty(M,(-\ep,\ep)) \to \on{Emb}(M,\R^d)$ is defined by $\ps_q(a) = q + an_q$ and $n_q$ is a unit-length normal vector field to $q$.

\subsection{Quotient representations of $B_e(M,\R^d)$}
Consider\footnote{The presentation follows \cite[Sect.~8]{Bauer2014}.} $B_e(M,\R^d)$ as the space of embedded type $M$ submanifolds of $\R^d$. We assume $\dim M < d$ for the space to be nonempty. For the remainder of this lecture we will write $\on{Diff}(\R^d)$ for the group\footnote{Since we are acting on embeddings of a compact manifold one could equally use $\on{Diff}_c(\R^d)$, the group of compactly supported diffeomorphisms, or $\on{Diff}_{\mc S}(\R^d)$, diffeomorphisms that decay rapidly towards the identity} $\on{Diff}_{H^\infty}(\R^d)$, in fact we only use the connected component of the identity of $\on{Diff}_{H^\infty}(\R^d)$, and $\mf X(M)$ for the corresponding space $\mf X_{H^\infty}(\R^d)$ of vector fields.

The natural action of $\on{Diff}(\R^d)$ on $B_e(M,\R^d)$ is given by
\[
\on{Diff}(\R^d) \x B_e(M,\R^d) \ni (\ph, Q) \mapsto \ph(Q) \in B_e(M,\R^d)\,.
\]
This action is in general not transitive -- consider for example a knotted and an unknotted loop in $\R^3$ -- but it is locally transitive and hence its orbits are open subsets of $B_e(M,\R^d)$. Since the group $\on{Diff}(\R^d)$ is connected, orbits of the $\on{Diff}(\R^d)$-action are the connected components of $B_e(M,\R^d)$. For $Q \in B_e(M,\R^d)$ the isotropy group
\[
\on{Diff}(\R^d)_Q = \left\{ \ph\,:\, \ph(Q) = Q \right\}\,,
\]
consists of all diffeomorphisms that map $Q$ to itself. Thus each orbit $\on{Orb}(Q) = \on{Diff}(\R^d).Q$ can be identified with the quotient
\[
B_e(M,\R^d) \supseteq \on{Orb}(Q) \cong {\on{Diff}(\R^d)}/{ \on{Diff}(\R^d)_Q}\,.
\]
Let us take a step backwards and remember that we defined $B_e(M,\R^d)$ to be the quotient
\[
B_e(M,\R^d) \cong \on{Emb}(M,\R^d) / \on{Diff}(M)\,.
\]
The diffeomorphism group $\on{Diff}(\R^d)$ also acts on the space $\on{Emb}(M,\R^d)$ of embeddings -- i.e., the space of parametrized submanifolds -- with the action
\[
\on{Diff}(\R^d) \x \on{Emb}(M,\R^d)\!\ni\! (\ph, q) \mapsto \ph \on{\circ} q \in \on{Emb}(M,\R^d).
\]
This action is generally not transitive either, but has open orbits as before. For fixed $q \in \on{Emb}(M,\R^d)$,  the isotropy group
\[
\on{Diff}(\R^d)_q = \left\{ \ph\,:\, \ph|q(M) \equiv \on{Id} \right\}\,,
\]
consists of all diffeomorphisms that fix the image $q(M)$ pointwise. Note the subtle difference between the two groups $\on{Diff}(\R^d)_q$ and $\on{Diff}(\R^d)_Q$, when $Q = q(M)$. The former consists of diffeomorphisms that fix $q(M)$ pointwise, while elements of the latter only fix $q(M)$ as a set. As before we can identify each orbit $\on{Orb}(q) = \on{Diff}(\R^d).q$ with the set
\[
\on{Emb}(M,\R^d) \supseteq \on{Orb}(q) \cong {\on{Diff}(\R^d)}/{ \on{Diff}(\R^d)_q}\,.
\]
The isotropy groups are subgroups of each other
\[
\on{Diff}(\R^d)_q \unlhd \on{Diff}(\R^d)_Q \leq \on{Diff}(\R^d)\,,
\]
with $\on{Diff}(\R^d)_q$ being a normal subgroup of $\on{Diff}(\R^d)_Q$. Their quotient can be identified with
\[
\on{Diff}(\R^d)_Q / \on{Diff}(\R^d)_q \cong \on{Diff}(M)\,.
\]
Now we have the two-step process,
\begin{multline*}
\on{Diff}(\R^d)
\to \on{Diff}(\R^d) / \on{Diff}(\R^d)_q \cong \on{Orb}(q) \subseteq \on{Emb}(M,\R^d) \to \\
 \to \on{Emb}(M,\R^d) / \on{Diff}(M) \cong B_e(M,\R^d)\,.
\end{multline*}
In particular the open subset $\on{Orb}(Q)$ of $B_e(M,\R^d)$ can be represented as any of the quotients
\begin{multline*}
\on{Orb}(Q) \cong \on{Orb}(q) / \on{Diff}(M) \cong \\
\cong \left.{}^{\displaystyle {\on{Diff}(\R^d)}/{ \on{Diff}(\R^d)_q} }
\middle/_{\displaystyle \on{Diff}(\R^d)_Q/\on{Diff}(\R^d)_q}\right. \cong
\on{Diff}(\R^d)/{ \on{Diff}(\R^d)_Q}\,.
\end{multline*}

\subsection{Metrics induced by $\on{Diff}_{H^\infty}(\R^d)$}
Let a right-invariant Riemannian metric $G^{\on{Diff}}$ on $\on{Diff}(\R^d)$ be given. Our goal is to define a metric on $\on{Emb}(M,\R^d)$ in the following way: fix an embedding $q_0 \in \on{Emb}(M,\R^d)$ and consider some other embedding $q = \ph \on{\circ} q_0$ in the orbit of $q_0$. Define the (semi-)norm of a tangent vector $h \in T_q \on{Emb}(M,\R^d)$ by
\[
G^{\on{Emb}}_q(h,h) = \inf_{X_\ph \on{\circ} q_0 = h} G^{\on{Diff}}_\ph(X_\ph, X_\ph)\,,
\]
with $X_\ph \in T_\ph\on{Diff}(\R^d)$. Intuitively we define the length of a tangent vector $h \in T_q \on{Emb}(M,\R^d)$ as the smallest length of a tangent vector $X_\ph$ inducing this infinitesimal deformation. If $\pi_{q_0}$ is the projection
\[
\pi_{q_0} : \on{Diff}(\R^d) \to \on{Emb}(M,\R^d)\,,\quad \pi_{q_0}(\ph) = \ph \on{\circ} q_0\,,
\]
then
\[
h = X_\ph \on{\circ} q_0 = T_\ph \pi_{q_0}.X_\ph\,,
\]
and the equation defining $G^{\on{Emb}}$ is the relation between two metrics that are connected by a Riemannian submersion,
\[
G^{\on{Emb}}_q(h,h) = \inf_{T_\ph\pi_{q_0}.X_\ph = h} G^{\on{Diff}}_\ph(X_\ph, X_\ph)\,.
\]
In fact the construction of $G^{\on{Emb}}_q$ depends neither on the diffeomorphism $\ph$ nor on the fixed embedding $q_0$. To see this note that
\[
X_\ph \circ q_0 = X_\ph \circ \ph\inv \circ \ph \circ q_0 = \left(X_\ph \circ \ph\inv\right) \circ q\,,
\]
and $X_\ph \circ \ph\inv \in T_{\on{Id}}\on{Diff}(\R^d)$. Hence we can write
\begin{equation}
\label{eq:def_G_emb_outer}
G_q^{\on{Emb}}(h, h) = \inf_{X \on{\circ} q = h} G^{\on{Diff}}_{\on{Id}}(X, X)\,,
\end{equation}
with $X \in T_{\on{Id}}\on{Diff}(\R^d)$. And this last equation depends only on $q$ and $h$.

One can show that the $G^{\on{Emb}}_q$ defined in this way is a positive semidefinite bilinear form. What is not obvious is that $G^{\on{Emb}}_q$ depends continuously or smoothly on $q$. This property and that it is positive definite, have to be checked in each example.

Assuming that this construction yields a (smooth) Riemannian metric on the space $\on{Emb}(M,\R^d)$, then this metric is invariant 
under $\on{Diff}(\R^d)$, because
the left-action by $\on{Diff}(\R^d)$ commutes with the right-action by $\on{Diff}(M)$:
\begin{multline*}
G^{\on{Emb}}_{q\on{\circ}\ph}(h\on{\circ}\ph,h\on{\circ}\ph) 
= \!\!\!\!\!\!\inf_{X\on{\circ} q\on{\circ}\ph = h\on{\circ}\ph}\!\!\!\!\!\! G^{\on{Diff}}_{\on{Id}}(X, X)
= \inf_{X\on{\circ} q = h} G^{\on{Diff}}_{\on{Id}}(X, X) = G^{\on{Emb}}_{q}(h,h)\,.
\end{multline*}
The metric $G^{\on{Emb}}$ then can be expected to project to a Riemannian metric on $B_e(M,\R^d)$.

\subsection{Existence of optimal lifts} Consider the metric
\[
G_{\on{Id}}^{\on{Diff}}(X,Y) = \int_{\R^d} \left\langle
\left(\on{Id} - \De\right)^n X, Y \right\rangle \!\ud x\,,
\]
and set $LX = (\on{Id} - \De)^n$. For $h \in T_q\on{Emb}(M,\R^d)$, how should an $X \in \mf X(\R^d)$ satisfying $X \circ q = h$ and
\[
G_q^{\on{Emb}}(h,h) = G_{\on{Id}}^{\on{Diff}}(X,X)
\]
look like? It has to satisfy $G_{\on{Id}}^{\on{Diff}}(X,Y) = 0$ for all $Y$ with $Y \circ q \equiv 0$. In other words
\[
\int_{\R^d} \langle LX, Y \rangle \ud x = 0\,,\quad 
\forall Y \in \mf X(\R^d) \text{ with } Y \circ q \equiv 0\,.
\]
Because $q(M)$ is a set of positive codimension and hence zero measure, there exists \emph{no smooth} function $LX$ satisfying this and therefore there exists no smooth $X$ attaining the infimum in \eqref{eq:def_G_emb_outer}. To find an infimum we have to look in a bigger space of less regular functions, for example we have hope to succeed if we allow $LX$ to be a distribution supported on the set $q(M)$.


\subsection{The RKHS point of view}

Let $(\mc H, \langle \cdot, \cdot \rangle_{\mc H})$ be a Hilbert space of vector fields, such that the canonical inclusions in the following diagram
\[
\mf X_{H^\infty}(\R^d) \hookrightarrow \mc H \hookrightarrow C^k_0(\R^d,\R^d)
\]
are bounded linear maps and $\mf X_{H^\infty}(\R^d)$ is dense in $\mc H$. We say that $\mc H$ is \emph{$k$-admissible}, if the inclusion $\mc H \hookrightarrow C^k_0$ is bounded. The motivation for the notion of $k$-admissible spaces of vector fields and their use to define groups of diffeomorphisms is explained in \cite{Younes2010}. 

The induced right-invariant metric on $\on{Diff}(\R^d)$ is
\[
G_{\ph}^{\on{Diff}}(X_\ph,Y_\ph) = \langle X_\ph\circ \ph\inv, Y_\ph\circ\ph\inv \rangle_{\mc H}\,,
\]
and the metric on $\on{Emb}(M,\R^d)$ is
\[
G^{\on{Emb}}_q(h,h) = \inf_{X \circ q = h} \langle X, X \rangle_{\mc H}\,.
\]

\begin{lemma*}
If the vector space $(\mc H, \langle \cdot, \cdot \rangle_{\mc H})$ is $0$-admissible, then the induced metric $G^{\on{Emb}}$ on $\on{Emb}(M,\R^d)$ is nondegenerate.
\end{lemma*}

\begin{proof}
Let $h \in T_q\on{Emb}(M,\R^d)$ and $x \in M$, such that $h(x) \neq 0$. Then
\[
h(x) \leq \| h \|_{\infty} = \|X \circ q\|_{\infty} \leq \| X \|_\infty \leq C \sqrt{\langle X, X, \rangle_{\mc H}}\,.
\]
Since this holds for all $X \in \mf X(\R^d)$ with $X \circ q = h$, we conclude that
\[
G^{\on{Emb}}_q(h,h) \geq C^{-2} |h(x)|^2 > 0\,,
\]
and hence the metric is nondegenerate.
\end{proof}

\subsection{The horizontal subspace}
To compute an explicit expression for $G^{\on{Emb}}_q$ we decompose $\mc H$ into
\begin{align*}
\mc H^{\on{ver}}_q &= \{X\in \mc H \,:\, X \circ q  \equiv 0 \}\,, &
\mc H^{\on{hor}}_q &= \left(\mc H^{\on{ver}}_q \right)^\perp\,.
\end{align*}
Note that since $\mc H$ is $0$-admissible, $\mc H^{\on{ver}}_q$ is a closed subspace and hence $\mc H^{\on{ver}}_q \oplus \mc H^{\on{hor}}_q = \mc H$. Then the induced metric is given by
\[
G^{\on{Emb}}_q(h,h) = \langle X^{\on{hor}}, X^{\on{hor}} \rangle_{\mc H}\,,
\]
where $X \in \mf X(\R^d)$ is any vector field satisfying $X \circ q = h$ and $X^{\on{hor}} \in \mc H^{\on{hor}}_q$ is its horizontal projection. The horizontal projection does not depend on the choice of the lift, i.e. if $X, Y \in \mf X(\R^d)$ coincide along $q$, then $X - Y \in \mc H^{\on{ver}}_q$ and hence $X^{\on{hor}} = Y^{\on{hor}}$.

We have the maps
\[
\begin{array}{ccc}
T_q \on{Emb}(M,\R^d) & \to & \mc H^{\on{hor}}_q \\
h & \mapsto & X^{\on{hor}}
\end{array}\,,
\qquad
\begin{array}{ccc}
\mc H^{\on{hor}}_q & \to & C^k(M,\R^d) \\
X & \mapsto & X\circ q
\end{array}\,.
\]
The composition of these two maps is the canonical embedding $T_q\on{Emb}(M,\R^d) \hookrightarrow C^k(M,\R^d)$. Because $M$ is compact we do not have to distinguish between $C^k$ and $C^k_b$. Furthermore the equation $G^{\on{Emb}}_q(h,h) = \langle X^{\on{hor}}, X^{\on{hor}}\rangle_{\mc H}$ shows that the first map is an isometry between $(T_q \on{Emb}(M,\R^d), G^{\on{Emb}}_q)$ and $\mc H^{\on{hor}}_q$.

\begin{lemma*}
The image of $T_q \on{Emb}(M,\R^d)$ is dense in $\mc H^{\on{hor}}_q$ and the $G^{\on{Emb}}_q$-com\-ple\-tion of $T_q \on{Emb}(M,\R^d)$ can be identified with $\mc H^{\on{hor}}_q$.
\end{lemma*}

\begin{proof}
It is enough to show that the image is dense. Given $X \in \mc H^{\on{hor}}_q$ choose a sequence $X_n \in \mf X_{H^\infty}(\R^d)$ converging to $X$ in $\mc H$. Then $X_n^{\on{hor}}$ is the image of $X_n \circ q \in T_q \on{Emb}(M,\R^d)$ in $\mc H^{\on{hor}}_q$ and
\[
\left\| X^{\on{hor}}_n - X \right\|_{\mc H} = \left\| X^{\on{hor}}_n - X^{\on{hor}} \right\|_{\mc H}
\leq \left\| X_n - X \right\|_{\mc H} \to 0\,.
\]
Hence the image is dense.
\end{proof}

A consequence of this lemma is that the $G^{\on{hor}}_q$-completion of $T_q \on{Emb}(M,\R^d)$ can be identified with a closed subspace of a RKHS and as such it is itself an RKHS. Since the norm $G^{\on{Emb}}_q$ is defined using the infimum $G^{\on{Emb}}_q(h,h) = \langle X^{\on{hor}}, X^{\on{hor}}\rangle_{\mc H}$, it follows from \cite[Thm.~I.5]{Aronszajn1950} that its reproducing kernel is given by restricting the kernel $K$ of $\mc Hx$ to $M$; i.e.,
\[
K_q : M \x M \to \R^{d \x d}\,,\quad K_q(x,y) = K(q(x),q(y))\,,
\]
is the reproducing kernel of the induced innner product on $T_q \on{Emb}(M,\R^d)$.

\appendix

\section{$\on{Diff}(M)$ as a Lie group}
Having discussed infinite-dimensional Riemannian manifolds, let us briefly look at an infinite-dimensional Lie group and some of its properties.

When $M$ is a manifold we can consider
\[
\on{Diff}(M) = \{ \ph \in C^\infty(M,M) \,:\, \ph \text{ bijective},\, \ph\inv \in C^\infty(M,M) \}\,,
\]
and $\mf X(M)$, the space of vector fields on $M$. Intuitively we would like to see $\on{Diff}(M)$ as a Lie group with $\mf X(M)$ as its Lie algebra and the Lie group exponential map being the time $1$ flow map of the vector field,
\[
\on{exp}: \mf X(M) \to \on{Diff}(M)\,,\quad u \mapsto \ph(1)\,,
\]
where $\ph(t)$ is the solution of the ODE
\[
\p_t \ph(t, x) = u(\ph(t,x))\,,\quad \ph(0,x) = x\,.
\]

This runs into several difficulties. First, \emph{$\on{exp}$ might not be well-defined}. Consider $M=\R$ and the vector field
\[
u(x) = x^2 \frac{\p}{\p x}\,.
\]
Its flow is given by
\[
\ph(t,x) = \frac{x}{1-tx}\,,
\]
and we see that the vector field $u$ is not complete. As a consequence $\ph(1,x)$ is defined only for $x < 1$. This problem can be avoided either by restricting to compact manifolds $M$ or by requiring vector fields to decay sufficiently rapidly towards infinity. But even then the exponential map exhibits some unexpected behaviour.

The exponential map \emph{is not locally surjective}. Consider an element $\ph \in \on{Diff}(S^1)$, such that $\ph$ has no fixed points and at least one isolated periodic point\footnote{A point $x \in S^1$, such that $\ph^n(x) = x$ for some $n \in \mb N_{>0}$.}. Assume that we can write $\ph = \on{exp}(u)$ for some $u \in \mf X(S^1)$. Then $u$ mus satisfy $u(x) \neq 0$ for all $x \in S^1$. Suprisingly, we can now show that $\ph$ is conjugate to a rotation. To see this define the diffeomorphism
\[
\et(x) = c \int_0^x \frac{\mathrm{d}y}{u(y)}\,,\quad
c = 2\pi \left(\int_{S^1} \frac{\mathrm{d}x}{u(x)} \right)\inv\,.
\]
It is easy to check that $\et \circ \ph \circ \et\inv$ is a rotation by calculating the derivative $\p_t\left(\et \circ \on{exp}(tu) \circ \et\inv\right)$. Now let $\et \circ \ph = R_\al \circ \et$, where $R_\al(x) = x + \al \mod 2\pi$ is a rotation. Then $\et \circ \ph^n = R_{n\al} \circ \et$. Let $x_0$ be the isolated periodic point such that $\ph^n(x_0) = x_0$. Then $\et(x_0) = R_{n\al}( \et(x_0))$, which implies that $R_{n\al} = \on{Id}_{S^1}$ and thus $\ph^n = \on{Id}_{S^1}$, which contradicts the assumption that $x_0$ is an isolated periodic point. Thus $\ph$ cannot be written as $\ph = \on{exp}(u)$ for any $u \in \mf X(S^1)$.
An example of such a $\ph$ is given by
\[
\ph(x) = x + \frac{2\pi}{n} + \ep \sin (nx)\,,
\]
where we can choose $n \in \mb N$ and $|\ep| < 2/n$. In particular, by choosing $n$ large and $\ep$ small, $\ph$ will be arbitrary close to the identity. This counterexample can be found in \cite[I.5.5.2]{Hamilton1982} and \cite[p.1017]{Milnor1984}. One can show more.

\begin{theorem*}[Grabowski, 1988 \cite{Grabowski1988}]
Given a $C^n$-manifold $M$, there exists a continuous curve $\ga : [0, 1) \to \on{Diff}^n_c(M)$, $\ga(0) = \on{Id}$, such that $\{ \ga(t) \,:\, t \in (0,1)\}$ is a set of free generators of a subgroup of $\on{Diff}^n_c(M)$, which contains only (apart from the identity) diffeomorphisms that are not in the image of the exponential map.
\end{theorem*}

Here $\on{Diff}^k_c(M)$ denotes the group of \emph{compactly supported} $C^n$-diffeomorphisms, i.e., $\ph$ is compactly supported if the set $\{ x \in M \,:\, \ph(x) \neq x \}$ has compact closure.

Finally the exponential map is \emph{not locally injective} either. Let $\ps \in \on{Diff}(S^1)$ be a $2\pi/n$-periodic diffeomorphism, i.e., $\ps(x + 2\pi/n) = \ps(x) + 2\pi/n$. Denote by $R_\al(x) = x + \al \!\!\mod 2\pi$ the rotation by $\al$. Then $R_{2\pi t/n}$ lies in the $1$-parameter subgroup $\ph(t) = \ps \circ R_{2\pi t/n} \circ \ps\inv$. In other words, define the vector field
\[
u(x) = \frac{2\pi}{n} \ps'(\ps\inv(x)) \frac{\p}{\p x}\,.
\]
Its flow is
\[
\ph(t,x) = \ps\left(\ps\inv(x) + \frac{2\pi}{n}t\right)\,,
\]
and $\on{exp}(u) = \ph(1) = R_{2\pi t/n}$. Since $\ps$ can be chosen to be arbitrary close to the identity, $\on{exp}$ cannot be locally injective.

\printbibliography

\end{document}